\documentclass[11pt]{amsart}
\usepackage{amsmath}
\usepackage{amsfonts}
\usepackage{amsthm}
\usepackage{amssymb}
\usepackage{amscd}
\usepackage[all]{xy}
\usepackage{enumerate}
\usepackage{color}
\usepackage{bm}
\usepackage{amstext}



\theoremstyle{plain}
\newtheorem{thm}{Theorem}[subsection]

\newtheorem{prop}[thm]{Proposition}

\newtheorem{cor}[thm]{Corollary}
\newtheorem{lem}[thm]{Lemma}

\theoremstyle{definition}
\newtheorem{defn}[thm]{Definition}

\newtheorem{gen}[thm]{Generality Condition}

\newtheorem{nota}[thm]{Notation}

\newtheorem*{ackn}{Acknowledgements}

\newtheorem{rem}[thm]{Remark}

\newcommand{\sE}{\mathcal{E}}
\newcommand{\sF}{\mathcal{F}}

\newcommand{\sH}{\mathcal{H}}
\newcommand{\sI}{\mathcal{I}}

\newcommand{\sL}{\mathcal{L}}

\newcommand{\sO}{\mathcal{O}}

\newcommand{\sR}{\mathcal{R}}
\newcommand{\sS}{\mathcal{S}}

\newcommand{\sU}{\mathcal{U}}

\newcommand{\mC}{\mathbb{C}}
\newcommand{\mF}{\mathbb{F}}

\newcommand{\mP}{\mathbb{P}}

\newcommand{\mZ}{\mathbb{Z}}

\newcommand{\Bs}{\mathrm{Bs}\,}

\newcommand{\Aut}{\mathrm{Aut}\,}

\newcommand{\Pic}{\mathrm{Pic}\,}

\newcommand{\PGL}{\mathrm{PGL}}
\newcommand{\Hspin}{\sS^{0,{\rm{hyp}}}_{g,1}}
\newcommand{\Pst}{(\mP^2)^*}
\newcommand{\tPst}{\widetilde{(\mP^2)^*}}
\newcommand{\fj}{{\sf{j}}}
\newcommand{\fm}{{\sf{m}}}
\newcommand{\tC}{{\empty^{\tau}\mC}}

\numberwithin{equation}{section}

\newenvironment{fcaption}{\begin{list}{}{
\setlength{\leftmargin}{35pt}
\setlength{\rightmargin}{35pt}
\setlength{\labelsep}{5pt}
}}{\end{list}}

\author{Hiromichi Takagi}

\address{Graduate School of Mathematical Sciences \\
the University of Tokyo\\
Tokyo, 153-8914, Japan\newline
\texttt{takagi@ms.u-tokyo.ac.jp}}

\author{Francesco Zucconi}

\address{D.I.M.I. \\
the University of Udine\\
Udine, 33100, Italy\newline
\texttt{Francesco.Zucconi@dimi.uniud.it}}

\begin{document}

\begin{center}
\textbf{\Large
The rationality of the moduli space of \\
one-pointed ineffective spin hyperelliptic curves \\
\vspace{4pt}
via an almost del Pezzo threefold}
\par
\end{center}
{\Large \par}

$\;$

\begin{center}
Hiromichi Takagi and Francesco Zucconi 
\par\end{center}

\vspace{5pt}

$\;$

\begin{fcaption} {\small  \item 

Abstract. 
Using the geometry of an almost del Pezzo threefold,
we show that the moduli space $\sS^{0,{\rm{hyp}}}_{g,1}$ of genus $g$ one-pointed ineffective spin hyperelliptic curves is rational for every $g\geq 2$.
}\end{fcaption}

\vspace{0.5cm}


\markboth{Takagi and Zucconi}{Moduli of one-pointed ineffective spin hyperelliptic curves}



\section{Introduction}
Throughout this paper, we work over $\mC$, the complex number field.
The purpose of this paper is to show the following result:

\begin{thm}[=Theorem \ref{belloiperellitticorat}] 
\label{tratto}
The moduli space $\sS^{0,{\rm{hyp}}}_{g,1}$ 
of one-pointed genus $g$ hyperelliptic ineffective spin curves is an irreducible rational variety.
\end{thm}

We have the following immediate corollary:

\begin{cor}
The moduli space $\sS^{0,{\rm{hyp}}}_{g}$  of genus $g$ hyperelliptic ineffective spin curves is an irreducible unirational variety.
\end{cor}

Now we give necessary definitions and notions to understand the statement of the above results.
We recall that a couple $(C,\theta)$ is called a {\it genus $g$ spin curve}
if $C$ is a genus $g$ curve and $\theta$ is a theta characteristic on $C$, namely, a half canonical divisor of $C$. If the linear system $|\theta|$ is empty, then $\theta$ is called an {\it ineffective} theta characteristic, and we also say
that such a {\it spin curve is ineffective}.
A {\it hyperelliptic spin curve} $(C, \theta)$ means that 
$C$ is hyperelliptic.
A pair of a spin curve $(C,\theta)$ and a point $p\in C$ is called 
a {\it one-pointed spin curve}.
One-pointed spin curves 
$(C,\theta, p)$ and
$(C',\theta', p')$ are said to be {\it isomorphic} to each other
if there exists an isomorphism $\xi\colon C \to C'$ such that
$\xi^*\theta'\simeq \theta$ and $\xi^*p'=p$.   
Finally, we denote by 
$\sS^{0,{\rm{hyp}}}_{g,1}$ (resp.~$\sS^{0,{\rm{hyp}}}_{g}$) the coarse moduli space of isomorphism classes of 
one-pointed genus $g$ hyperelliptic ineffective spin curves
(resp.~genus $g$ hyperelliptic ineffective spin curves).

Main motivations of our study are the rationalities of the moduli spaces of
hyperelliptic curves \cite{Bo} and of pointed hyperelliptic curves \cite{Ca}. 

One feature of the paper is that 
the above rationality is proved via the geometry of a certain smooth projective threefold. We developed such a method in our previous works \cite{TZ1, TZ2, TZ3}. In these works, we established the interplay
between 
\begin{itemize}
\item
even spin trigonal curves, where  even spin curve means that 
the considered theta characteristics have even-dimensional spaces
of global sections, and 
\item 
the quintic del Pezzo threefold $B$, which is known to be unique up to isomorphisms and is isomorphic to a codimension three linear section of $\mathrm{G}(2,5)$.
\end{itemize}

The relationship between curves and $3$-folds are a kind of mystery but 
many such relationships have been known to nowadays. A common philosophy of such works is
that a family of certain objects in a certain threefold is an algebraic curve
with some extra data.
In \cite[Cor.~4.1.1]{TZ1}, we showed that a genus $d-2$ trigonal curve appears 
as the family of lines on $B$ which intersect a fixed another rational curve of degree $d\geq 2$,
and, in \cite[Prop.~3.1.2]{TZ2}, we constructed a theta characteristic on the trigonal curve from
the incidence correspondence of intersecting lines on $B$. 
The mathematician who met first 
such an interplay is S.~Mukai, who discovered that lines on a genus twelve prime Fano threefold $V$ is parameterized by a genus three curve, and constructed 
a theta characteristic on the the genus three curve from
the incidence correspondence of intersecting lines on $V$ \cite{Mu2, Mukai12}. In our previous works \cite{TZ1,TZ2,TZ3},
we interpreted Mukai's work from the view point of the quintic del Pezzo threefold $B$ and generalized it. 

The study of this paper is directly related to 
our paper \cite{TZ3}, in which we showed that the moduli of even spin genus four curves is rational by using the above mentioned interplay. 

We are going to show our main result also by using such an interplay, but we replace
the quintic del Pezzo threefold by a certain degeneration of it. This is a new feature of this paper. The degeneration is a quintic del Pezzo threefold with one node, which is also
known to be unique up to isomorphisms and is isomorphic to a codimension three linear section of $\mathrm{G}(2,5)$ by \cite{Fu3}. Moreover, it is
not factorial at the node, and hence it admits two small resolutions, which we call $B_a$ and $B_b$ in this paper. Actually, we do not work on 
this singular threefold directly but work on small resolutions, mainly on $B_a$.
Along the above mentioned philosophy, we consider a family of `lowest degree' rational curves on $B_a$, which we call $B_a$-lines, intersecting a fixed another `higher degree' rational curve $R$. Then we show such $B_a$-lines are parameterized by a
hyperelliptic curve $C_R$, and 
we construct an ineffective theta characteristic $\theta_R$
on it from the incidence correspondence of intersecting $B_a$-lines.
Then we may reduce the rationality problem of the moduli 
to that of a certain quotient of family of rational curves on $B_a$
by the group acting on $B_a$, and 
solve the latter by computing invariants. 

Finally, we sketch the structure of the paper.  
In the section \ref{section:Ba}, we define a projective threefold $B_a$,
which is the key variety for our investigation of one-pointed ineffective
spin hyperelliptic curves. In this section, we also review several properties
of $B_a$.
In the section \ref{section:iperellittico}, we construct the above mentioned families of rational curves $R$ on $B_a$, and the family of $B_a$-lines.
Then, in the section \ref{Hyp},
we construct hyperelliptic curves $C_R$
as the parameter space of $B_a$-lines intersecting each fixed $R$.
In the section \ref{section:theta}, we construct
an ineffective theta characteristic $\theta_R$ on $C_R$
from the incidence correspondence of intersecting $B_a$-lines
parameterized by $C_R$. We also remark that
$C_R$ comes with a marked point from its construction. Finally in this section,
we interpret the moduli $\Hspin$ by a certain group quotient of the family of $R$. Then, in the section \ref{section:Rat}, we show the rationality of the latter by computing invariants.

\begin{ackn}
The authors thank Yuri Prokhorov for very useful conversations about the topic.
This research is supported by MIUR funds, 
PRIN project {\it Geometria delle variet\`a algebriche} (2010), coordinator A. Verra (F.Z.), and, by Grant-in Aid for Young Scientists (B 20740005, H.T.)
and by Grant-in-Aid for Scientific Research (C 16K05090, H.T.). 
\end{ackn}


\section{The key projective threefold $B_a$}
\label{section:Ba}
\subsection{Definition of $B_a$}
The key variety to show the rationality of $\Hspin$ is the threefold, which we denote by $B_a$ in this paper, with the following properties:
\begin{enumerate}[(1)]
\item
$B_a$ is a smooth {\it almost del Pezzo threefold}, which is, by definition, a smooth projective threefold with nef and big but non-ample anticanonical divisor divisible by $2$ in the Picard group. 
\item
If we write $-K_{B_a}=2M_{B_a}$, then $M^3_{B_a}=5$.
\item
$\rho(B_a)=2$.
\item
$B_a$ has two elementary contractions, one of which is the anticanonical
model $B_a\to B$ and it is a small contraction, and another is 
a $\mP^1$-bundle $\pi_a\colon B_a\to \mP^2$. 
\end{enumerate}

\subsection{Descriptions of $B_a$}
\label{sub:DesBa}
Many people met the threefold $B_a$ in several contexts. 
The first one is probably T.~Fujita.
In his classification of singular del Pezzo threefolds \cite{Fu3},
$B_a$ appears as a small resolution of the quintic del Pezzo threefold $B$.
Here we do not review Fujita's construction of $B_a$ in detail
except that we sum up his results as follows:
\begin{prop}
$B_a$ is unique up to isomorphism, and
the anti-canonical model $B_a\to B$ contracts a single smooth rational curve, say, $\gamma_a$ to a node of $B$.
 In particular the normal bundle of $\gamma_a$ is $\sO_{\mP^1}(-1)^{\oplus 2}$.
\end{prop}

Fujita treats $B_a$ less directly, so descriptions of $B_a$ 
by \cite{L}, \cite{JPR05}, \cite{Take} and \cite{Hu}, which we review below, are more convenient for our purpose.

By \cite[\S 3]{L} and \cite[Thm.~3.6]{JPR05}, we may write $B_a\simeq \mP(\sE)$ with
a stable rank two bundle $\sE$ on $\mP^2$ with $c_1(\sE)=-1$ and $c_2(\sE)=2$
fitting in the following exact sequence:
\begin{equation}
\label{eq:E}
0\to \sO(-3)\to \sO(-1)^{\oplus 2}\oplus \sO(-2)\to \sE\to 0.
\end{equation}
Let $H_{\sE}$ be the tautological divisor for $\sE$ and $L$ the $\pi_a$-pull back of a line in $\mP^2$.
By the canonical bundle formula for projective bundle,
we may write $-K_{B_a}=2H_{\sE}+4L$.
Therefore, by the definition of $M_{B_a}$, 
we see that $M_{B_a}$ is the tautological line bundle associated to $\sE(2)$.

Generally,
let $\sF$ be a stable bundle on $\mP^2$ with $c_1(\sF)=-1$.
In \cite{Hu}, Hulek studies jumping lines 
for such an $\sF$, where a line $\fj$ on $\mP^2$ is called a {\it jumping line} for $\sF$ 
if $\sF_{|{\fj}}\not \simeq \sO_{\mP^1}\oplus \sO_{\mP^1}(-1)$. We also recall that
a line $\sf{l}$ on $\mP^2$ is called a {\it jumping line of
the second kind} for $\sF$ if $h^0(\sF_{|{2\sf{l}}})\not =0$.
In [ibid.~Thm.~3.2.2], it is shown that
the locus $C(\sF)$ in the dual projective plane $\Pst$
parameterizing jumping lines of the second kind
is a curve of degree $2(c_2(\sF)-1)$.
Therefore, in our case, $C(\sE)$ is a conic. Moreover
the following properties of $\sE$ hold by [ibid.]:

\begin{prop}
\label{prop:jump}
\begin{enumerate}[$(1)$]
\item $\sE$ is unique up to an automorphism of $\mP^2$,
\item $C(\sE)\subset \Pst$ is a line pair, 
which we denote by $\ell_1\cup \ell_2$, 
\item $\sE$ has a unique jumping line $\subset \mP^2$, which we denote by $\fj$,
and the point $[\fj]$ in the dual projective plane $\Pst$ is equal to 
$\ell_1\cap \ell_2$, and
\item
$\sE_{|{\fj}}\simeq \sO_{\mP^1}(-2)\oplus \sO_{\mP^1}(1)$. 
\end{enumerate}
\end{prop}

\begin{proof}
(1)--(3) follow from [ibid.~Prop.~8.2], 
and (4) follows from [ibid.~Prop.~9.1].
\end{proof}

\begin{nota}
For a line $\fm\subset \mP^2$, we set $L_{\fm}:=\pi_a^{-1}(\fm)\subset B_a$.
We denote by $C_0(\fm)$ the negative section of $L_{\fm}$.
\end{nota}

Here we can interpret the jumping line of $\sE$ 
by the birational geometry of $B_a$
as follows:

\begin{cor}
\label{cor:flop}
The $\pi_a$-image on $\mP^{2}$ of the exceptional curve $\gamma_a$ of $B_a\to B$
is the jumping line $\fj$.
\end{cor}

\begin{proof}
By the uniqueness of $\gamma_a$, we have only to show that
the negative section $C_0(\fj)$ of $L_{\fj}$ is numerically trivial for $-K_{B_a}$.
By Proposition \ref{prop:jump} (4), we have $H_{\sE}\cdot C_0(\fj)=-2$.
Therefore, since $-K_{B_a}=2H_{\sE}+4L_{\fj}$, we have
$-K_{B_a}\cdot C_0(\fj)=2\times (-2)+4 =0$.
\end{proof}

\subsection{Two-ray link}    
By \cite[Thm.~3.5 and 3.6]{JPR05} and \cite[Thm.~2.3]{Take},
a part of the birational geometry of $B_a$ is described by the following two-ray link: 
\begin{equation}
\label{eq:ODP}
\xymatrix{
& B_a \ar[dl]_{\pi_{a}}\ar[dr]  &\dashrightarrow & 
 B_b \ar[dl] \ar[dr]^{\pi_{b}} & \\
 \mP^{2}  &  & B & & \mP^1, }
\end{equation}
where

\begin{enumerate}[(i)]
\item
$B_a\dashrightarrow B_b$ is the flop of a single smooth rational curve $\gamma_a$.

\item
$\pi_{b}$ is a quadric bundle.
\item
Let $L$ be the pull-back of a line by $\pi_{a}$, and $H$ a fiber of 
$\pi_{b}$.
Then 
\begin{equation}
\label{eq:twice}
-K=2(H+L), 
\end{equation}
where we consider this equality both on $B_a$ and $B_b$, and $-K$ 
denotes both of the anti-canonical divisors.
\end{enumerate}

\begin{nota}
\begin{enumerate}[(1)]
\item
We denote by $\gamma_a$ and $\gamma_b$ the flopping curves on $B_a$ and $B_b$,
respectively.
\item
It is important to notice that there exist exactly two singular $\pi_b$-fibers,
which are isomorphic to the quadric cone
(this follows from the calculation of the topological Euler number of $B_a$
and invariance of Euler number under flop).
We denote them by $F_1$ and $F_2$.
\end{enumerate}
\end{nota}

Though we mainly work on $B_a$,
the threefold $B_b$ is also useful to understand 
the properties of $B_a$ related to the jumping lines of the second kind
since the definition of such jumping lines is less geometric (see the subsection \ref{sub:Bb}).

\subsection{Group action on $B_a$}
In this subsection, we show that $B_a$ has a natural action
by the subgroup of $\Aut \Pst$ fixing $\ell_1 \cup \ell_2$ . 
This fact should be known for experts but we do not know appropriate literatures.

Our way to see this is based on the elementary transformation of the $\mP^2$-bundle $\pi_a\colon B_a\to \mP^2$ centered at the flopping curve $\gamma_a$.
This make it possible to describe the group action quite explicitly.
\begin{prop}
\label{prop:elm}
Let $\mu\colon \widetilde{B}_a\to B_a$ be the blow-up along the flopping curve
$\gamma_a$. Let $\nu\colon \widetilde{B}_a\to B_c$ be the blow down over $\mP^2$ contracting the strict transform of $L_{\fj}=\pi_a^{-1}(\fj)$ to a smooth rational curve $\gamma_c$ $($the existence of the blow down follows from Mori theory in a standard way$)$.
Then $B_c\simeq \mP^1\times \mP^2$. Moreover, $\gamma_c$ is a divisor of type $(1,2)$ in $\mP^1\times \fj$.
\begin{equation}
\label{eq:elm}
\xymatrix{& \widetilde{B}_a\ar[dl]_{\mu}\ar[dr]^{\nu} &\\
B_a\ar[dr]_{\pi_a}  &  & B_c\ar[dl]\\ 
 & \mP^{2} &  }
\end{equation}
\end{prop}

\begin{proof}
This follows from \cite[p.166, (si111o) Case (a)]{Fu3}.
\end{proof}


Let $(x_1:x_2)$ be a coordinate of $\mP^1$ and 
$(y_1:y_2:y_3)$ be a coordinate of $\mP^2$.
By a coordinate change, 
we may assume that $\fj=\{y_3=0\}\subset \mP^2$ and 
the two ramification points of 
$\gamma_c\hookrightarrow \mP^1\times \mP^2\overset{p_1}{\to} \mP^1$ are 
$(0:1)\times (1:0:0)$ and $(1:0)\times (0:1:0)$.
Then $\gamma_c=\{\alpha x_1 y_1^2+\beta x_2 y_2^2=y_3=0\}$
with $\alpha\beta\not=0$. By a further coordinate change, 
we may assume that 
\begin{equation}
\label{gammaC}
\gamma_c=\{x_1 y_1^2+x_2 y_2^2=y_3=0\}.
\end{equation}

Let us denote by $G$ the automorphism group of $B_a$.
Now we can easily obtain the following description of $G$ from Proposition \ref{prop:elm}.
For this, we denote by $G_m\simeq \mC^*$ the multiplicative group and
by $G_a\simeq \mC$ the additive group.
\begin{cor}
\label{cor:elm}
The automorphism group $G$ of $B_a$ is isomorphic to 
the subgroup of the automorphism group of $B_c$ which preserves $\gamma_c$.
Explicitly, let an element
$(A, B)\in \PGL_2\times \PGL_3$ acts on $B_c\simeq \mP^1\times \mP^2$
as $({\bf x},{\bf y})\mapsto (A{\bf x}, B{\bf y})$ by matrix multiplication.
If $(A,B)$ preserve $\gamma_c$ with the equation $(\ref{gammaC})$ as above, then $(A,B)$ is of the form
\begin{enumerate}[$(\rm i)$]
\item $A={\scriptsize \begin{pmatrix} a_1^2 & 0\\ 0 & 1\end{pmatrix}},\,
B={\scriptsize \begin{pmatrix} 1& 0 & b_1\\
0 & a_1 & b_2 \\
0 & 0 & a_2 
\end{pmatrix}}$,
or
\item
$A={\scriptsize \begin{pmatrix} 0 & a_1^2 \\ 1 & 0\end{pmatrix}},\,
B={\scriptsize \begin{pmatrix} 0& 1 & b_1\\
a_1 & 0 & b_2 \\
0 & 0 & a_2 
\end{pmatrix}},$
\end{enumerate}
where $a_1, a_2\in G_m$ and $b_1,b_2\in G_a$ in both cases.

In particular, 
the $G$-orbit of $(1:1)\times (0:0:1)$ in $\mP^1\times \mP^2$ is open.
Therefore, the action of $G$ on $B_c$ is, and hence the one on $B_a$ is quasi-homogeneous. 
\end{cor}

It is also easy and is convenient to write down the $G$-action on the base $\mP^2$.

\begin{cor}
\label{cor:gr}
\begin{enumerate}[$(1)$]
\item
The projective plane $\mP^2$ consists of the following three orbits of $G:$
\[
\mP^2=G\cdot (0:0:1)\sqcup G\cdot (1:1:0)\sqcup \{(1:0:0)\sqcup (0:1:0)\},
\]
where $G\cdot (0:0:1)$ is the open orbit, 
$G\cdot (1:1:0)$ is an open subset of the jumping line $\fj:=\{y_3=0\}$,
and the two points $(1:0:0), (0:1:0)\in \fj$ form one orbit and correspond 
to the lines $\ell_1$ and $\ell_2$ by projective duality.
\item
The dual projective plane $\Pst$ has the following three orbits of $G$ by the contragredient action of $G:$\[
\Pst=G\cdot (1:1:0)\sqcup \{G\cdot (1:0:0)\sqcup G\cdot (0:1:0)\}\sqcup (0:0:1),
\]
where 
$G\cdot (1:1:0)$ is the open orbit, 
the closures of $G\cdot (1:0:0)$ and $G\cdot (0:1:0)$
are the two lines $\ell_1$ and $\ell_2$.
\end{enumerate}

\end{cor}

\begin{proof}
We only show that 
the two points $(1:0:0), (0:1:0)\in \fj$ correspond 
to the lines $\ell_1$ and $\ell_2$ by projective duality.
This follows from the orbit decomposition of $\mP^2$
by the identity component $G_0$ of $G$ 
since the two points $\in \mP^2$ corresponding to
the lines $\ell_1$ and $\ell_2$ are fixed by $G_0$, and 
$G_0$ has only two fixed points.
\end{proof}

In the section \ref{section:Rat},
a central role is played by the following explicit description of the action of $G$ on $B_a$
preserving $L_{\fm}$ for a general $\fm$. By quasi-homogenousity of the action on $B_a$, we may assume that $\fm=\{y_1=y_2\}$.
 
\begin{lem}\label{lem:fixLm}
An element $(A, B)\in \PGL_2\times \PGL_3$ of $G$ preserves $L_{\fm}$, equivalently, preserves $\fm$ if and only if 
$(A,B)$ is of the form
\[
(a) \ A={\scriptsize \begin{pmatrix} 1 & 0\\ 0 & 1\end{pmatrix}},\,
B={\scriptsize \begin{pmatrix} 1& 0 & b_1\\
0 & 1 & b_1 \\
0 & 0 & a_2 
\end{pmatrix}},
\]
or
\[
(b)\ A={\scriptsize \begin{pmatrix} 0 & 1 \\ 1 & 0\end{pmatrix}},\,
B={\scriptsize \begin{pmatrix} 0& 1 & b_1\\
1 & 0 & b_1 \\
0 & 0 & a_2 
\end{pmatrix}},
\]
where $a_2\in G_m$ and $b_1\in G_a$ in both cases.

In particular, such elements form a subgroup $\Gamma \simeq (\mZ_2\times G_a)\rtimes G_m$ and $\Gamma$ is generated by the following three type elements\,$:$
\begin{itemize}
\item $G_m:$
{\scriptsize{$\begin{pmatrix} 1 & 0\\ 0 & 1\end{pmatrix}\times
\begin{pmatrix} 1& 0 & 0\\
0 & 1 & 0\\
0 & 0 & a 
\end{pmatrix}$}}
with $a\in G_m$,
\item $G_a:$
{\scriptsize $\begin{pmatrix} 1 & 0\\ 0 & 1\end{pmatrix}\times
\begin{pmatrix} 1& 0 & b\\
0 & 1 & b\\
0 & 0 & 1 
\end{pmatrix}$}
with $b\in G_a$, and
\item $\mZ_2:$
\scriptsize{$\begin{pmatrix} 0& 1\\ 1 & 0\end{pmatrix}\times
\begin{pmatrix} 0& 1 & 0\\
1 & 0 & 0\\
0 & 0 & 1 
\end{pmatrix}.$}

\end{itemize}
\end{lem}


\section{Families of rational curves on $B_a$}
\label{section:iperellittico}
In this section, we construct families of rational curves on $B_a$,
which will ties the geometries of $B_a$ and one-pointed ineffective spin hyperelliptic curves. We start by some preliminary discussions.

\begin{lem}
\label{lem:neg}
If a line $\fm$ is not equal to the jumping line $\fj$, then $(H-L)_{|{L_{\fm}}}$ is linearly equivalent to the negative section $C_0({\fm})$ of $L_{\fm}\simeq \mF_1$.
If $\fm=\fj$, then $(H-L)_{|{L_{\fm}}}$ is 
linearly equivalent to the negative section $C_0(\fj)$ plus a ruling of $L_{\fj}\simeq \mF_3$.
\end{lem}

\begin{proof}
As we mention in the subsection \ref{sub:DesBa},
$M_{B_a}$ is the tautological line bundle on $B_a$ associated to the bundle $\sE(2)$. Therefore, by (\ref{eq:twice}), $H-L=M_{B_a}-2L$ is
the tautological line bundle associated to the bundle $\sE$.
If $\fm$ is not equal to the jumping line $\fj$, then $\sE_{|{\fm}}\simeq \sO_{\mP^1}\oplus \sO_{\mP^1}(-1)$ and hence $(H-L)_{|{L_{\fm}}}$ is linearly equivalent to the negative section $C_0({\fm})$ of $L_{\fm}\simeq \mF_1$.
If $\fm=\fj$, then $\sE_{|{\fm}}\simeq \sO_{\mP^1}(-2)\oplus \sO_{\mP^1}(1)$ and hence $(H-L)_{|{L_{\fm}}}$ is 
linearly equivalent to the negative section $C_0(\fj)$ plus a ruling of $L_{\fj}\simeq \mF_3$.
\end{proof}

By this lemma, it is easy to show the following proposition:

\begin{prop}\label{razional 1}
Let $\fm\subset \mP^2$ be a line and $g\geq -1$ an integer.
If $\fm\not =\fj$ $($resp.~$\fm=\fj$ and $g\geq 1)$, 
then a general element $R$ of the linear system $|(H+gL)_{|L_{\fm}}|$
 is 
a smooth rational curve with $H\cdot R=g+1$ and $L\cdot R=1$.
Moreover,
if $\fm\not =\fj$ and $g\geq 0$ $($resp.~$\fm=\fj$ and $g\geq 1)$, 
then $|(H+gL)_{|L_{\fm}}|$ has no base point.
\end{prop}

\begin{defn}
We define $\sL$ to be the following subvariety of $B_a\times (\mP^2)^*$:
\[
\sL:=\{(x,[\fm])\mid x\in L_{\fm}=\pi_a^{-1}(\fm)\}.
\]
Let $p_1\colon  \sL \to B_a$ and 
$p_2\colon \sL\to (\mP^2)^*$ be the first and the second projections, respectively. Note that the $p_2$-fiber over a point $[\fm]$ is nothing but $L_{\fm}$.
\end{defn}

\begin{rem}
\label{rem:GonL}
To follow the sequel easily, it is useful to notice that
$\sL$ is the pull-back by the composite $B_a\times \Pst\overset{\pi_a\times \rm{id}}{\longrightarrow} \mP^2\times \Pst$ of the point-line incidence variety
$\{(x,[\fm])\mid x\in \fm\}\subset \mP^2\times \Pst$.
Therefore, we also see that $\sL$ is $G$-invariant, where the $G$-action is induced on $B_a\times \Pst$ by the $G$-action on $B_a$ defined as above and 
the contragredient $G$-action on $\Pst$. 
\end{rem}

\subsection{Higher degree case}
\label{higher}

\begin{defn}\label{proclaimer}
\begin{enumerate}[(1)]
\item 
For an integer $g\geq 0$,
we set 
\[
\sR_g:=p_{2*} p_1^*\sO_{B_a}(H+gL).
\]  
We see that $\dim H^0(\sO_{L_{\fm}}(H+gL))$ is constant since 
$H^1(\sO_{L_{\fm}}(H+gL))=\{0\}$ for any $\fm$ and $g\geq 0$. 
Therefore, by Grauert's theorem, $\sR_g$ is a locally free sheaf on $(\mP^2)^*$.
Set 
\[
\Sigma_g:=\mP(\sR_g^*),
\]
which is nothing but
the projective bundle over $(\mP^2)^*$
whose fiber over a point $[\fm]$ is the projective space $\mP(H^0(\sO_{L_{\fm}}(H+gL)))$.

\item
We denote by $\sH_g\subset \Sigma_g$ the sublocus
parameterizing smooth rational curves.
Note that
$\sH_g$ is a non-empty open subset of $\Sigma_g$
by Proposition \ref{razional 1}.
\end{enumerate}
\end{defn}

\subsection{$B_a$-Lines}
Now we construct a family of curves parameterizing the negative section of $L_{\fm}$ for an $\fm\not= \fj$, and the negative section plus a ruling of $L_{\fj}$. Intuitively, it is easy to imagine such a family but a rigorous construction needs some works. 

\begin{lem}
\label{lem:comp}
The following hold\,$:$
\begin{enumerate}[$(1)$]
\item
$H^0(\sO_{B_a}(H-L))=\{0\}$ and $H^1(\sO_{B_a}(H-L))=\mC$.
\item
$H^0(\sO_{B_a}(H-2L))=\{0\}$,
$H^1(\sO_{B_a}(H-2L))=\mC^2$, and
$H^2(\sO_{B_a}(H-2L))=\{0\}$.
\end{enumerate}
\end{lem}

\begin{proof} 
The results follow easily from the exact sequence (\ref{eq:E}).
Here we only show that $H^2(\sO_{B_a}(H-2L))\simeq H^2(\mP^2,\sE(-1))=\{0\}$,
By the Serre duality, we have  
$H^2(\mP^2,\sE(-1))\simeq H^0(\mP^2, \sE^*(-2))^*\simeq
H^0(\mP^2, \sE(-1))^*$, which is zero by (\ref{eq:E}). 
\end{proof}

\begin{nota}
\label{nota:blup}
Let $b\colon \tPst\to \Pst$ be the blow-up at the point $[\fj]$.
Let $E_0$ be the $b$-exceptional curve, and 
$r$ be a ruling of $\tPst\simeq \mF_1$.
The surface $\tPst$ will be the parameter space of the family of rational curves
which we are going to construct.  

For a point $[\fm]\in \Pst\setminus [\fj]$,
we use the same character $[\fm]$
for the corresponding point on $\tPst$.
\end{nota}

Let $b_{\sL}\colon \widetilde{\sL}\to \sL$ be the blow-up along the fiber of $p_2\colon \sL\to \Pst$ over $[\fj]$. By universality of blow-up, the variety $\widetilde{\sL}$ is contained in $B_a\times \tPst$ and a unique map $\tilde{p}_2\colon \widetilde{\sL}\to \tPst$ is induced.
We denote by $\tilde{p}_1\colon \widetilde{\sL}\to B_a$ 
the map obtained by composing $\widetilde{\sL}\to \sL$ with $p_1\colon \sL\to B_a$. 
\begin{equation}
\label{eq:sq}
\xymatrix{
\widetilde{\sL}\ar[r]^{b_{\sL}}\ar[d]_{\tilde{p}_2}\ar@/^18pt/[rr]^{\tilde{p}_1}& \sL\ar[d]_{p_2}\ar[r]^{p_1}& B_a\\
\tPst\ar[r]^{b} & \Pst &}
\end{equation}

\begin{lem}
\label{lem:mC}
It holds that $H^0(\tilde{p}_1^*\sO_{B_a}(H-L)\otimes \tilde{p}_2^*\sO(E_0+2r))\simeq \mC$.
\end{lem}

\begin{proof}
Let $\tilde{\rho}_1\colon  B_a\times \tPst \to B_a$ and
$\rho_1\colon  B_a\times \Pst \to B_a$ be the first projections, and 
$\tilde{\rho}_2\colon B_a\times \tPst \to \tPst$ and
$\rho_2\colon B_a\times \Pst \to \Pst$ 
the second projections.
By Remark \ref{rem:GonL}, as a divisor on $B_a\times (\mP^2)^*$,
$\sL$ is linearly equivalent to $\rho_1^*L+\rho_2^*\sO_{\Pst}(1)$.
Since $\sL$ does not contain the fiber of $B_a\times (\mP^2)^*\to \Pst$
over $[\fj]$, the variety $\widetilde{\sL}$ is the total pull-back of $\sL$ by 
$B_a\times \tPst\to B_a\times \Pst$. Hence 
$\widetilde{\sL}$ is linearly equivalent to 
$\tilde{\rho}_1^*L+\tilde{\rho}_2^*(E_0+r)$ since $\sO(E_0+r)=b^*\sO_{\Pst}(1)$.

Now let us consider the following exact sequence:
\begin{align*}
0\to \widetilde{\rho}_1^*\sO_{B_a}(H-2L)\otimes \widetilde{\rho}_2^*\sO(-E_0-r)\to \\
\widetilde{\rho}_1^*\sO_{B_a}(H-L)\to
\widetilde{p}_1^*\sO_{B_a}(H-L)\to 0,
\end{align*}
which is obtained from the natural exact sequence
\[
0\to \sO_{B_a\times \tPst}(-\widetilde{\sL})\to\sO_{B_a\times \tPst}\to\sO_{\widetilde{\sL}}\to 0.
\]
by tensoring $\widetilde{\rho}_1^*\sO_{B_a}(H-L)$.
By Lemma \ref{lem:comp}, the pushforward of the exact sequence by $\widetilde{\rho}_2$ is 
\begin{align*}
0\to \tilde{p}_{2*}\tilde{p}_1^*\sO_{B_a}(H-L)\to \sO(-E_0-r)^{\oplus 2}\to \sO\to 
R^1 \tilde{p}_{2*}\tilde{p}_1^*\sO_{B_a}(H-L)\to 0.
\end{align*}
Note that, for a point $[\fm]\not =[\fj]$, it holds that $H^0(\sO_{L_{\fm}}(H-L))\simeq \mC$ and $H^1(\sO_{L_{\fm}}(H-L))=\{0\}$ by Lemma \ref{lem:neg}.
Therefore, by Grauert's theorem,
$\tilde{p}_{2*}\tilde{p}_1^*\sO_{B_a}(H-L)$ is an invertible sheaf possibly outside $E_0$, and the support of $R^1:=R^1 \tilde{p}_{2*}\tilde{p}_1^*\sO_{B_a}(H-L)$ is contained in $E_0$.

We show that the support of $R^1$ is equal to $E_0$.
Indeed, let $\sI$ be the image of the map $\sO(-E_0-r)^{\oplus 2}\to \sO$
in the above exact sequence,
which is an ideal sheaf. Then the closed subscheme $\Delta$ defined by $\sI$ is
the intersection of one or two members of $|E_0+r|$.
In particular, $\Delta$ is non-empty.
Noting $\sO_{\Delta}=R^1$ and the support of $R^1$ is contained in $E_0$,
the subscheme $\Delta$ must be equal to $E_0$.  

Therefore, the map $\sO(-E_0-r)^{\oplus 2}\to \sO$ is decomposed as
$\sO(-E_0-r)^{\oplus 2}\to \sO(-E_0)\hookrightarrow \sO$
and $\sO(-E_0-r)^{\oplus 2}\to \sO(-E_0)$ is surjective. 
Hence the kernel $\tilde{p}_{2*}\tilde{p}_1^*\sO_{B_a}(H-L)$ 
of the map $\sO(-E_0-r)^{\oplus 2}\to \sO(-E_0)$ is isomorphic to
$\sO(-E_0-2r)$.
Now we can compute
\begin{align*}
H^0(\widetilde{p}_1^*\sO_{B_a}(H-L)\otimes \widetilde{p}_2^*\sO(E_0+2r))\simeq\\
H^0(\widetilde{p}_{2*}\widetilde{p}_1^*\sO_{B_a}(H-L)\otimes \sO(E_0+2r))\simeq\\
H^0(\sO(-E_0-2r)\otimes \sO(E_0+2r))\simeq \mC.
\end{align*}
\end{proof}

In the next proposition, we obtain the desired family of curves. 

\begin{prop}
\label{prop:U1}
Let $\sU_1$ be the unique member of $|\tilde{p}_1^*\sO_{B_a}(H-L)\otimes \tilde{p}_2^*\sO(E_0+2r)|$.
Then the natural map $\sU_1\to \tPst$ is flat. Moreover, the fibers are described as follows\,$:$
\begin{enumerate}[$(1)$]
\item
the fiber over a point $[\fm]\not =[\fj]$ is the negative section of $L_{\fm}$,
and   
\item
the fiber over a point $x$ of $E_0$ 
is the negative section plus a ruling of $L_{\fj}$.
\end{enumerate}
\end{prop}

\begin{proof}
Note that $\sU_1$ is Cohen-Macaulay since it is a divisor on a smooth variety.
Therefore the flatness follows from the smoothness of $\tPst$ and 
the descriptions of fibers, which we are going to give below.

Note that, by the uniqueness of $\sU_1$, the group $G$ acts on $\sU_1$,
where $G$ acts on $\sL$ and hence on $\widetilde{\sL}$ by Remark \ref{rem:GonL}.
Let $x$ be a point of $\tPst$. Set $[\fm]:=b(x)\in \Pst$.
Note that the fiber of $\widetilde{\sL}\to \tPst$ over $x$ is 
$L_{\fm}$. 

If $x\not \in E_0$, then $L_{\fm}\subset\sU_1$ or ${\sU_1}_{|{L_{\fm}}}$ is the negative section of $L_{\fm}\simeq \mF_1$
since
$\sU_1\in |\tilde{p}_1^*\sO_{B_a}(H-L)\otimes \tilde{p}_2^*\sO(E_0+2r)|$. We show that the latter occurs for any $x\not \in E_0$, which implies the assertion (1). 
If $L_{\fm}\subset \sU_1$ for an $x\not \in E_0$
such that $[\fm]\not \in \ell_1\cup \ell_2$,
then, by the description of the group action of $G$ (Corollary \ref{cor:gr}),
$L_{\fm}\subset \sU_1$ hold for all such $x$'s, which implies that
$\sU_1=\widetilde{\sL}$, a contradiction.
If $L_{\fm}\subset \sU_1$ for an $x\not \in E_0$
such that $[\fm]\in \ell_i$ for $i=1,2$,
then, again by the group action of $G$,
$L_{\fm}\subset \sU_1$ hold for all such $x$'s, which implies that
$\sU_1$ contains the pull-back of the strict transform $\ell'_i\subset \tPst$
of $\ell_i$. Since $\ell'_i$ is a ruling of $\tPst$,
this implies that 
$H^0(\tilde{p}_1^*\sO_{B_a}(H-L)\otimes \tilde{p}_2^*\sO(E_0+r))\not =0$,
which is impossible by the proof of Lemma \ref{lem:mC}.

Now assume that $x\in E_0$.
By a similar argument to the above one using the group action, we see that ${\sU_1}_{|{L_{\fj}}\times \{x\}}$ is the negative section plus a ruling
if $x$ is not contained in the strict transforms $\ell'_i$ of $\ell_i$
($i=1,2$). Therefore 
${\sU_1}_{|{L_{\fj}}\times E_0}$ is a member of the linear system $|\sO_{L_{\fj}}(H-L)\boxtimes \sO_{E_0}(1)|$ on ${L_{\fj}}\times E_0$. Suppose by contradiction that 
${\sU_1}_{|{L_{\fj}}\times \{x\}}={L_{\fj}}\times \{x\}$ for $x=\ell'_1\cap E_0$ or $\ell'_2\cap E_0$. Then, since the group action interchanges
$\ell'_1\cap E_0$ and $\ell'_2\cap E_0$, 
${\sU_1}_{|{L_{\fj}}\times \{x\}}={L_{\fj}}\times \{x\}$ for both $x=\ell'_1\cap E_0$ and $\ell'_2\cap E_0$.
This would imply
that $|\sO_{L_{\fj}}(H-L)\boxtimes \sO_{E_0}(-1)|$ is nonempty,
which is absurd.  
 
Therefore the assertion (2) follows.
\end{proof}

\begin{defn}
We call a fiber of $\sU_1\to \tPst$ {\it a $B_a$-line}.
Explicitly, by Proposition \ref{prop:U1},
a $B_a$-line is the negative section $C_0(\fm)$ of $L_{\fm}$ for $[\fm]\not =[\fj]$, or
the negative section $C_0(\fj)$ plus a ruling of $L_{\fj}$.

The name comes from the fact that the image of a $B_a$-line on the anti-canonical model $B$ is a line in the usual sense
when $B$ is embedded by $|M_B|$,
where $M_B$ is the ample generator of $\Pic B$.

\end{defn}

\subsection{$B_a$-Lines interpreted on $B_b$}
\label{sub:Bb}

In the section \ref{Hyp}, we will construct hyperelliptic curves using the map $\tilde{p}_{1|\sU_1} \colon \sU_1\to B_a$.
To understand the map $\tilde{p}_{1|\sU_1}$, 
it is convenient to interpret $B_a$-lines by the geometry of $B_b$.

\begin{nota}
\label{nota:F1F2}
\begin{enumerate}[(1)]
\item 
We denote by $F_1$ and $F_2$ the two singular $\pi_b$-fibers and 
by $v_i$ the vertex of $F_i$ ($i=1,2$).
\item
We denote by $F'_i$ the strict transform on $B_a$ of $F_i$ ($i=1,2$).

\item
By Corollary \ref{cor:flop}, we have $L\cdot \gamma_a=1$, and, by a standard property of flop, we have $L\cdot \gamma_b=-1$. This and the equality (\ref{eq:twice}) imply that $\gamma_b$ is a $\pi_b$-section.
Therefore $\gamma_b$ does not pass through $v_1$ nor $v_2$ and so
$B_b\dashrightarrow B_a$ is isomorphic near $v_1$ and $v_2$.
We denote by $v'_i$ the point on $B_a$ corresponding to $v_i$ $(i=1,2)$.
\end{enumerate}
\end{nota}

\begin{prop}
\label{prop:onBb}
\begin{enumerate}[$(1)$]
\item
The $\pi_a$-images of $v'_1$ and $v'_2$ in $\mP^2$ correspond to
the lines $\ell_1$ and $\ell_2$ in $\Pst$ by projective duality.
\item For a line $\fm\not =\fj$ on $\mP^2$,
the negative section $C_0(\fm)$ of $L_{\fm}$ is disjoint from $\gamma_a$.
\item
For a line $\fm\not =\fj$ on $\mP^2$, the curve $C_0(\fm)$ is the strict transform  of a ruling of a $\pi_b$-fiber disjoint from $\gamma_b$, and vice-versa.
Moreover, under this condition, $C_0(\fm)$ is the strict transform of a ruling of $F_1$ or $F_2$ if and only if $[\fm]\in \ell_1\cup \ell_2$. 
\item
A ruling $f$ of $L_{\fj}$ is the strict transform of a ruling of a $\pi_b$-fiber intersecting $\gamma_b$, and vice-versa $($note that $f\cap \gamma_a\not =\emptyset$, and $\gamma_a\cup f$ is a $B_a$-line$)$.
Moreover, under this condition, $f$ is the strict transform of a ruling of $F_1$ or $F_2$ if and only if the point $\pi_a(f)\in \mP^2$ corresponds to the line $\ell_1$ or $\ell_2$ in $\Pst$ by projective duality.
\end{enumerate}
\end{prop} 

\begin{proof}
We show the assertion (1).
We use the group actions of $G$ on $B_a$ and $B_b$. The action of $G$ on $B_b$ fixes or interchanges $F_1$ and $F_2$, and hence $v_1$ and $v_2$. 
Since $B_b\dashrightarrow B_a$ is isomorphic near $v_1$ and $v_2$ as we noted 
in Notation \ref{nota:F1F2}, the group action on $B_a$ fixes or interchanges $v'_1$ and $v'_2$.
By Corollary \ref{cor:gr},
this implies that the images of $v'_1$ and $v'_2$ correspond to
the lines $\ell_1$ and $\ell_2$ by projective duality.

We show the assertion (2).
Let $C'_0(\fm)$ be the strict transform of $C_0(\fm)$ on $B_b$.
Note that $H\cdot C_0(\fm)=0$ by Lemma \ref{lem:neg}.
If $C_0(\fm)\cap \gamma_a\not =\emptyset$,
then $H\cdot C'_0(\fm)<H\cdot C_0(\fm)=0$ by a standard property of flop, which is a contradiction since
$H$ is nef on $B_b$.

We show the first assertions of (3) and (4).
Since the proofs are similar, we only show (4), which is more difficult.
We also only prove the only if part since the if part follows by reversing the argument. 
Recall that $\gamma_a+f \sim (H-L)_{|{L_{\fj}}}$.
Thus $H\cdot f=1$. Since $f$ intersects $\gamma_a$ transversely at one point,
and $H\cdot \gamma_a=-1$, 
we have 
$H\cdot f'=0$, where $f'$ is the strict transform of $f$ on $B_b$. Hence $f'$ is contained in a $\pi_b$-fiber $F$.
By the equality (\ref{eq:twice}), we have $-K_F=-K_{B_b}|_F=2L|_F$.
Therefore $f'$ is a ruling of $F$ since $L\cdot f'=L\cdot f+1=1$.

The latter assertions of (3) and (4) follows from (1).
\end{proof}

\begin{cor}
\label{cor:double}
Let $x$ be a point of $B_a\setminus \gamma_a$. 
If $x$ is not in the strict transform of $F_1$ nor $F_2$, then 
$x$ is contained in exactly two $B_a$-lines.
If $x$ is in the strict transform of $F_1$ or $F_2$
and is not equal to $v'_1$ nor $v'_2$, then 
$x$ is contained in exactly one $B_a$-line.

In particular, outside $\gamma_a\cup v'_1\cup v'_2$,
the map $\tilde{p}_{1|\sU_1}\colon \sU_1\to B_a$ is finite of degree two and is branched along $F'_1$ and $F'_2$.
\end{cor}

\begin{proof}
The assertions follow from Proposition \ref{prop:onBb}
(3) and (4), and the description of rulings on
quadric surfaces.  
\end{proof}


\section{Hyperelliptic curves parameterizing $B_a$-lines}
\label{Hyp}

\begin{defn}
\label{defn:CR}
Let $\fm\subset \mP^2$ be a line and $R\subset L_{\fm}$
a (not necessarily irreducible) member of the linear system $|(H+gL)_{|{L_{\fm}}}|$ (cf.~the subsection \ref{higher}). 
\begin{enumerate}[(1)]
\item We define 
\[
C_R:=\tilde{p}_1^{-1}(R)\cap \sU_1\subset \widetilde{\sL}
\]
with the notation as in the diagram (\ref{eq:sq}).
Here we take the intersection scheme-theoretically as follows\,$:$
first we consider $\tilde{p}_1^{-1}(L_m)\cap \sU_1$,
which is a divisor in $\sU_1$.
Second, we consider $C_R=\tilde{p}_1^{-1}(R)\cap \sU_1$ as a divisor
in $\tilde{p}_1^{-1}(L_{\fm}) \cap \sU_1$.
\item
We define $\widetilde{M}_R$ to be the image of $C_R$ on $\tPst$, and 
$M_R$ to be the image of $C_R$ on $\Pst$.
Note that $\widetilde{M}_R$ parameterizes $B_a$-lines intersecting $R$.
\item For a $\pi_a$-fiber $f$,
we also define
$C_f$, $\widetilde{M}_f$, and $M_f$
in a similar fashion to (1) and (2). 
\end{enumerate}
\end{defn}

In Proposition \ref{trovato} below,
we are going to show that $C_R$ is a hyperelliptic curve of genus $g$
under the following generality conditions for $\fm$ and $R$ as in Definition \ref{defn:CR}:

\begin{gen}
\label{gen}
Let $\fm\subset \mP^2$ be a line and $R\subset L_{\fm}$
a member of the linear system $|(H+gL)_{|{L_{\fm}}}|$. 
We consider the following conditions for $\fm$ and $R$:

\vspace{5pt}

\begin{enumerate}[(a)]
\item
$[\fm]\not \in \ell_1\cup \ell_2$. In particular, 
$v'_1, v'_2\not \in R$ by Proposition \ref{prop:onBb} (1).
\item
$R$ is smooth.
\item $R\cap \gamma_a=\emptyset$.
\item
$R$ intersects $F'_1$ and $F'_2$ transversely at $g+1$ points, respectively
(note that, by $R\sim (H+gL)_{|{L_{\fm}}}$,
we have $F'_i\cdot R=H\cdot R=g+1$).

Note that the condition (c)
implies that
$R\cap F'_1\cap F'_2=R\cap \gamma_a=\emptyset$.
\end{enumerate}

\vspace{5pt}

It is easy to see that, if $g\geq 0$, then general $\fm$ and $R$ satisfy
these conditions by Proposition \ref{razional 1}.
\end{gen}

\begin{lem}
\label{lem:F'}
If $[\fm]\not \in \ell_1\cup \ell_2$, then ${F'_i}_{|{L_{\fm}}}$ is
linearly equivalent to $C_0(\fm)+L_{|{L_{\fm}}}$, and is irreducible $(i=1,2)$.
In particular, $C_0(\fm)$ is disjoint from $F'_i$. 
\end{lem}

\begin{proof}
We see that ${F'_i}_{|{L_{\fm}}}$ is
linearly equivalent to $C_0(\fm)+L_{|{L_{\fm}}}$ by
(\ref{eq:twice}) since $F'_i\sim H$.
Assume by contradiction that ${F'_i}_{|{L_{\fm}}}$ is reducible.
Then $C_0(\fm)\subset {F'_i}_{|{L_{\fm}}}$.
Therefore, by Proposition \ref{prop:onBb} (3),
$C_0(\fm)$ passes through $v'_i$, which contradicts the assumption that 
$[\fm]\not \in \ell_1\cup \ell_2$.
\end{proof}

\begin{lem}
\label{gen2}
Assume that a $\pi_a$-fiber $f$ is disjoint from $\gamma_a$.
Then the following hold\,$:$
\begin{enumerate}[$(1)$]
\item 
$v'_1, v'_2\not\in f$.
\item
$f$ intersects $F'_1$ and $F'_2$ at one point, respectively,
and
$f\cap F'_1\cap F'_2=\emptyset$.

\end{enumerate}
\end{lem}

\begin{proof}
The assumption $f\cap \gamma_a=\emptyset$ is equivalent to that
$\pi_a(f)$ belongs to the open orbit of $G$. Therefore the assertion (1) follows from Corollary \ref{cor:gr} and Proposition \ref{prop:onBb} (1).

We show that $f$ intersects $F'_i$ $(i=1,2)$ at one point.
By (\ref{eq:twice}), we have $F'_i\cdot f=H\cdot f=1$ since $-K_{B_a}\cdot f=2$ and $L\cdot f=0$. Therefore, we have only to show that $f$ is not contained in $F'_i$. If $f\subset F'_i$, then the strict transform $f'$ of $f$ is contained in $F_i$. Then, however, $-K_{F_i}\cdot f'=2L_{|F_i}\cdot f'=0$, a contradiction.

The assumption implies that
$f\cap F'_1\cap F'_2=f\cap \gamma_a=\emptyset$.
Therefore we have the assertion (2).

\end{proof}

\begin{prop}\label{trovato}
Assume that $g\geq 2$, and
$\fm$ and $R$ satisfy Generality Condition $\ref{gen}$ $(\rm{a})$--$(\rm{d})$.  
Then the following hold\,$:$
\begin{enumerate}[$(1)$]
\item $C_R$ is a smooth hyperelliptic curve of genus $g$.
The hyperelliptic structure is given by the map $\tilde{p}_{1|{C_R}}\colon C_R\to R$ and the map is branched at $R\cap (F'_1\cup F'_2)$.

\item
Assume that a $\pi_a$-fiber $f$ is disjoint from $\gamma_a$.
Then $C_f$ is a smooth rational curve and $C_f\to f$ is a double cover branched at
the two points $f\cap (F'_1\cup F'_2)$. Moreover, $M_f$ is the line of $\Pst$
corresponding to the point $\pi_a(f)\in \mP^2$ by projective duality.

\item $\widetilde{M}_R\subset \tPst$ is a curve
which is smooth outside the point $[\fm]$ and 
has a $g$-ple point at $[\fm]$.
\item
$\deg M_R=g+2$ and $M_R\simeq \widetilde{M}_R$.
Moreover, the unique $g^1_2$ on $C_R$ is given by the pull-back of 
the pencil of lines through $[\fm]$.  
\end{enumerate}
\end{prop}

\begin{proof} 
We use the notation in Section \ref{section:iperellittico} freely.

\vspace{2pt}

\noindent(1). 
By Corollary \ref{cor:double} and Generality Condition \ref{gen} (a)--(d),
we see that $C_R$ with reduced structure satisfies
all of the claimed properties. Therefore, we have only to show $C_R$ is reduced. 
It suffices to show this for a general $R$ since $C_R$ form a flat family for $R$'s with Generality Condition \ref{gen} (a)--(d).
By the Bertini theorem on $\sL$, 
the divisor $\tilde{p}_1^{-1}(L_m)\cap \sU_1$ is a reduced surface since $|L|$ has no base point.
Now, again by the Bertini theorem in $\tilde{p}_1^{-1}(L_{\fm}) \cap\, \sU_1$,
the intersection $C_R=\tilde{p}_1^{-1}(R)\cap \sU_1$ is also reduced, and we are done. 

\vspace{2pt}

\noindent(2). The assertions for $C_f$ can be proved similarly to (1) by Lemma \ref{gen2}.
As for $M_f$, note that $f$ intersects the negative sections of $L_{\fm}$'s such that $\fm\ni \pi_a(f)$. Therefore the assertion follows since $M_f$ is the image of the smooth curve $C_f$. 

\vspace{5pt}

To show the remaining assertions, we investigate fibers of $\tilde{p}_1^{-1}(L_{\fm})\cap \sU_1\to \tPst$.
For a point $s\in \tPst$, the fiber over $s$ is the intersection between
$L_{\fm}$ and the $B_a$-line corresponding to $s$.
Therefore the fiber over $[\fm]$ can be identified with the negative section $C_0(\fm)$ of $L_{\fm}$.
Recall that $E_0$ is as in Notation \ref{nota:blup}. 
Let $t$ be the point of $E_0$ 
over which the fiber of $\sU_1\to \tPst$ is the union of $\gamma_a$ and the ruling of $L_{\fj}$ over $\fj\cap \fm$. Then the fiber of 
$\tilde{p}_1^{-1}(L_{\fm})\cap \sU_1\to \tPst$ over $t$ is
the ruling of $L_{\fj}$ over $\fj\cap \fm$.
Besides,
over $\tPst\setminus
([\fm]\cup t)$, the map $\tilde{p}_1^{-1}(L_{\fm})\cap \sU_1\to \tPst$ is one to one, hence is an isomorphism by the Zariski main theorem.

Note that the map $\sU_1\to \tPst$ is smooth over $\tPst\setminus E_0$.
Therefore,
$\tilde{p}_1^{-1}(L_{\fm})\cap \sU_1$ is smooth possibly outside the fiber over $t$. 
Therefore $\tilde{p}_1^{-1}(L_{\fm})\cap \sU_1\to \tPst$ is the blow-up at 
$[\fm]$ near $[\fm]$. We denote by $E_{\fm}$ and $E_t$ the exceptional curves over 
$[\fm]$ and $t$, respectively.

\vspace{5pt}

\noindent (3).
By Lemma \ref{lem:F'}, $\tilde{p}_1^{-1}(C_0(\fm))\cap \sU_1\to C_0(\fm)$ is an \'etale double cover, and hence $\tilde{p}_1^{-1}(C_0(\fm))\cap \sU_1$ consists of two disjoint smooth rational curves. Since $E_{\fm}\subset \tilde{p}_1^{-1}(C_0(\fm))\cap \sU_1$,
we see that $E_{\fm}$ is one of its components by the above description of $E_{\fm}$. 
Thus, since $R$ intersects $C_0(\fm)$ transversely at $g$ points,
the curve $C_R$ intersects $E_{\fm}$ transversely at $g$ points. 
Therefore we have the assertion (3) by blowing down $E_{\fm}$, except that we postpone proving
$\widetilde{M}_R$ is smooth outside $[\fm]$ in (4).

\vspace{2pt}

\noindent(4). 
To compute $\deg M_R$, we take a general fiber $f$ of $L_{\fm}\to \fm$
such that $f\cap \gamma_a=\emptyset$ and 
$R\cap f \not \in F'_1\cup F'_2\cup C_0(\fm)$. 
Then, since  $R\cap f \not \in F'_1\cup F'_2$, 
$C_R$ and $C_f$ intersect transversely at two points,
which is the inverse image of one point $R\cap f$.
Since $R\cap f \not \in C_0(\fm)$,
$C_R$ and $C_f$ does not intersect on $E_{\fm}$.
Therefore, the intersection multiplicity of $M_R$ and $M_f$ at $[\fm]$ is $g$. Thus we conclude that $\deg M_R=M_R\cdot M_f=g+2$ since $M_f$ is a line by generality of $f$ and the assertion (2).

Now the facts that $\widetilde{M}_R$ is smooth outside $[\fm]$ and
 $\widetilde{M}_R\simeq M_R$ 
follow since $g(C_R)=g$, $\deg M_R=g+2$ and $M_R$ has a $g$-ple point at $[\fm]$.

\end{proof}

\begin{rem}[Hyperelliptic structure of $C_R$ via the geometry of $B_b$]
\label{rem:hypBb}
Using the interpretation of $B_a$-lines on $B_b$, we may describe 
the hyperelliptic structure of $C_R$ on $B_b$. We think that it is helpful for
the readers to bear this in mind, so we give a sketch of it.

By Proposition \ref{prop:onBb} (3) and (4), $B_a$-lines correspond to
rulings of $\pi_b$-fibers in a one to one way.
Thus we may identify $C_R$ with the relative Hilbert scheme $H_R$ of 
rulings of $\pi_b$-fibers. The natural map $H_R\to \mP^1$, where $\mP^1$ is the target of $\pi_b$, is a double cover branched at the images of
two singular $\pi_b$-fibers since 
a smooth quadric has two families of rulings while a singular quadric has
one such a family.
\end{rem}

\begin{nota}[{\bf The marked point $[\fj]_R$ on $C_R$}]
\label{nota:marked}
Assume that $R$ satisfies Generality Condition \ref{gen} (a)--(d).
Then there is a unique $B_a$-line intersecting $R$ of the form $\gamma_a$ plus a ruling,
which is $\gamma_a\cup (L_{\fm}\cap L_{\fj})$.
We denote by $[\fj]_R$ the point of the hyperelliptic curve $C_R$ corresponding to this $B_a$-line since
this point is mapped to $[\fj]\in M_R\subset \Pst$. 
\end{nota}


\section{Theta characteristics on the hyperelliptic curves.}
\label{section:theta}
\subsection{Constructing theta characteristics}

By the above understanding of the hyperelliptic double cover $C_R\to R$,
we may construct an ineffective theta characteristic on
$C_R$ as follows: 

\begin{prop}\label{thetaiperdone} For a curve $R$ satisfying Generality Condition $\ref{gen}$ $(\rm{a})$--$(\rm{d})$ and $g\geq 2$, we denote by $h_R$
 the unique $g^1_2$ on the hyperelliptic curve $C_R$ $($cf.~Proposition $\ref{trovato}$ $(1)$ and $(4))$. Let $\nu\colon C_R\to M_R$ be the morphism constructed in Proposition $\ref{trovato}$, which is the normalisation.
Then  
$$\sO_{C_R}(\theta_R):=\nu^*\sO_{M_R}(1)\otimes_{\sO_{C_R}}\sO_{C_R}(-h_R-[\fj]_R)$$
is an ineffective theta characteristic on $C_R$.
\end{prop}

\begin{proof}
Let $F$ be one of the two singular $\pi_b$-fibers 
and $F'$ its strict transform on $B_a$. 
By Generality Condition $\ref{gen}$ (d), $R$ intersects $F'$ transversely at $g+1$ points, which we denote by $s_1,\dots, s_{g+1}$.
By Proposition \ref{trovato} (1), 
these points are contained in the branched locus of the hyperelliptic double cover $C_R\to R$. We denote by $t_1,\dots, t_{g+1}$
the inverse images on $C_R$ of $s_1,\dots, s_{g+1}$, and
by $u_1,\dots, u_{g+1}$ the images on $M_R$ of $t_1,\dots, t_{g+1}$.
Then, by Proposition \ref{prop:onBb} (3) and (4),
$u_1,\dots, u_{g+1}$ are contained in $\ell:=\ell_1$ or $\ell_2$.
We show that the points $u_1,\dots, u_{g+1}$
are different from $[\fj]$. 
Note that the unique $B_a$-line through a point $u_i$
is the strict transform $l_i$ of a ruling of $F$,
or the union of $l_i$ and $\gamma_a$ by 
Proposition \ref{prop:onBb} (3) and (4).
Assume by contradiction that
$u_i=[\fj]$ for some $i$.
Then, by Proposition \ref{prop:U1}, the latter occurs, namely, 
$l_i\cap \gamma_a\not =\emptyset$ and $l_i$ is a $B_a$-fiber.
Moreover, by Proposition \ref{prop:onBb} (4), 
the point $\pi_a(l_i)\in \mP^2$ corresponds to $\ell\subset \Pst$ by projective duality.
This implies that $[\fm]\in \ell$,
a contradiction to Generality Condition \ref{gen} (a). 

Therefore, since $\ell$ and $M_R$ contain $[\fj]$, and $\deg M_R=g+2$, we have 
$\ell_{|{M_R}}=u_1+\cdots+u_{g+1}+[\fj]$. 
Then, by the definition of $\theta_R$,
we have $\theta_R=t_1+\cdots+t_{g+1}-h_R$.
Now the assertion follows from \cite[p.288, Exercise 32]{ACGH}.
\end{proof}

\begin{rem}
\begin{enumerate}[(1)]
\item
In the proof of Proposition \ref{thetaiperdone}, we obtain the presentation
$\theta_R=t_1+\cdots+t_{g+1}-h_R$. So there are two such presentations according to
choosing $\ell_1$ or $\ell_2$. This is compatible with 
\cite[p.288, Exercise 32 (ii)]{ACGH}.
\item
In the introduction, we say that we construct the theta characteristic
from the incidence correspondence of intersecting $B_a$-lines.
We add explanations about this since this is not obvious from the above construction.  

The flow of the consideration below is quite similar to
the proof of Proposition \ref{thetaiperdone}.
Instead of a singular $\pi_b$-fiber, we consider a smooth $\pi_b$-fiber
$H\simeq \mP^1\times \mP^1$. Let $r_1$ and $r_2$ be the two rulings of $H$ intersecting $\gamma_b$, and $r'_1$ and $r'_2$ the strict transforms on $B_a$ of $r_1$ and $r_2$, respectively. By Proposition \ref{prop:onBb} (4), 
$r'_1$ and $r'_2$ are two $\pi_a$-fibers such that $\pi_a(r'_1), \pi_a(r'_2)\in \fj$.
Let $\delta_1$ and $\delta_2$ are the families of rulings of $F$ containing $r_1$ and $r_2$ respectively. By Proposition \ref{prop:onBb} (3) and (4),
there exists a family $\delta'_i$ of $B_a$-lines corresponding to $\delta_i$ $(i=1,2)$.
Note that this is nothing but $\widetilde{M}_{r_{3-i}}$ defined as in
Definition \ref{defn:CR} (3).
By the same proof as that of Proposition \ref{trovato} (2), we see that 
$M_{r_{3-i}}$ is a line in $\Pst$.
Let $l_1,\dots, l_{d-1}\in \delta'_1$ be the $B_a$-lines intersecting $R$.
Note that the $B_a$-line $r'_1\cup \gamma_a$ intersects $r'_2$ and
corresponds to the point $[\fj]$. Therefore 
$[\fj], [l_1],\dots, [l_{d-1}]\in M_R\cap M_{r_2}$.
In a similar way to the proof of 
Proposition \ref{thetaiperdone}, we can show that
$[\fj]$ is different from $[l_1],\dots, [l_{d-1}]$.
Hence we have $M_{r_2}\cap M_R=[l_{1}]+\ldots +[l_{d-1}]+[l]$. Let $m_1,\dots, m_{d-1}\in \delta'_2$ be the $B_a$-lines intersecting $R$.
By relabelling if necessary, we have $h_R\sim [l_i]+[m_i]$ by Remark \ref{rem:hypBb}.
Choose one of $m_i$'s, say, $m_1$.
Then, by the definition of $\theta_R$,
we have $\theta_R+[m_1]=[l_2]+\cdots+[l_{d-1}]$.
The $B_a$-lines $l_2,\dots, l_{d-2}$ are nothing but those intersecting $m_1$ and $R$
($l_1$ is excluded since it will  be disjoint from $m_1$
after the blow-up along $R$. See \cite[\S 4]{TZ1} and \cite[\S 3.1]{TZ2} for this consideration).  
\end{enumerate}
\end{rem}

\subsection{Reconstructing rational curves} 
Let $g\geq 2$. By Propositions \ref{trovato}, and \ref{thetaiperdone}
(see also Notation \ref{nota:marked}), 
we obtain a rational map
\begin{equation}\label{thetaiperdonemap}
    \pi_{g,1}\colon\sH_{g+2}
    \dashrightarrow\sS^{0,{\rm{\tiny{hyp}}}}_{g,1},\,\, [R]\mapsto 
    [C_R, [\fj]_R, \theta_R],
\end{equation}
which is fundamental for our purpose.
 
The next theorem shows how to construct the rational curve $R$ such that $\pi_{g,1}([R])=[(C,p,\theta)]$ for a general element 
$[(C,p,\theta)]$ in $\sS^{0,{\rm{hyp}}}_{g,1}$. 

This is one of our key result to show the rationality of $\sS^{0,{\rm{hyp}}}_{g,1}$.
 

\begin{thm}{\bf{(Reconstruction theorem)}}\label{Spinhyperellipticloci} 
The map $\pi_{g,1}$ is dominant. 
More precisely, let $[(C,p,\theta)]\in 
    \sS^{0,{\rm{hyp}}}_{g,1}$ be any element
such that $p$ is not a Weierstrass point, then there exists a point $[R]\in\sH_{g+2}$ such that $R$ satisfies Generality Condition $\ref{gen}$ $(\rm{a})$--$(\rm{d})$ and
    $\pi_{g,1}([R])=[(C,p,\theta)]$.
\end{thm}

For our proof of the theorem, we need the following general results for
an element of $\sS^{0,{\rm{hyp}}}_{g,1}$. The proof given below is slightly long but it is elementary and only uses standard techniques from algebraic curve theory.
\begin{lem}
\label{lem:reconst}
Let $[(C,p,\theta)]$ be any element of $\sS^{0,{\rm{hyp}}}_{g,1}$.
Let $\{p_1,\dots, p_{g+1}\}\cup \{p'_{1},\dots, p'_{g+1}\}$ be 
the partition of the set of 
the Weierstrass points of $C$
such that
$\theta$ has the following two presentations\,$:$
\begin{equation}
\label{eq:twopre}
\theta\sim p_1+\cdots+p_{g+1}-g^1_2\sim p'_{1}+\cdots+p'_{g+1}-g^1_2
\end{equation}
$($cf.~\cite[p.288, Exercise 32]{ACGH}$)$.
The following assertions hold\,$:$
\begin{enumerate}[$(1)$]
\item The linear system $|\theta+g^{1}_{2}+p|$ defines a birational morphism
from $C$ to a plane curve of degree $g+2$. 
\item
$|\theta+p|$ has a unique member $D$ and 
it is mapped to a single point $t$ by the map $\varphi_{|\theta+g^{1}_{2}+p|}$.
\vspace{5pt}

\noindent For the assertions $(3)$ and $(4)$, we set $S:=\{p, p_1,\dots, p_{g+1}, p'_{1},\dots, p'_{g+1}\}$.
\item
The support of $D$ contains no point of $S$.
\item
The point $t$ as in $(2)$ is different from the $\varphi_{|\theta+g^{1}_{2}+p|}$-images of points of $S$. 
Besides, by the map $\varphi_{|\theta+g^{1}_{2}+p|}$,
no two points of $S$ are mapped to the same point. 
\end{enumerate}
\end{lem}

\begin{proof} 
\noindent (1). We show that the linear system $|\theta+g^{1}_{2}+p|$ has no base points. 
By (\ref{eq:twopre}),
we see that $\Bs |\theta+g^{1}_{2}+p|\subset \{p\}$.
By the Serre duality, we have 
\[
H^1(\theta+g^{1}_{2}+p)\simeq H^0(K_C-\theta-g^1_2-p)^*=H^0(\theta-g^1_2-p)=0
\]
since $\theta$ is ineffective.
Similarly, we have $H^1(\theta+g^{1}_{2})=\{0\}$.
Therefore, by the Riemann-Roch theorem,
\[
h^0(\theta+g^{1}_{2}+p)-h^0(\theta+g^{1}_{2})=
\chi(\theta+g^{1}_{2}+p)-\chi(\theta+g^{1}_{2})=1,
\]
which implies that $p\not \in \Bs |\theta+g^{1}_{2}+p|$.

By the above argument, we see that $h^0(\theta+g^{1}_{2}+p)=\deg
(\theta+g^{1}_{2}+p)+1-g=3$. Therefore, $|\theta+g^{1}_{2}+p|$ gives a morphism $\varphi_{|\theta+g^{1}_{2}+p|}\colon C 
    \to\mP(V)\simeq \mP^2$ with $V=H^0(C,\sO_C(\theta+g^{1}_{2}+p))^*$.
Let $M:=\varphi_{|\theta+g^{1}_{2}+p|}(C)$ be the image of $C$.
We show that $C\to M$ is birational. Note that by the Riemann-Roch theorem
and $h^1(\theta+p)=h^0(K-\theta-p)=0$,
we have $h^0(\theta+p)=1$.
Therefore the hyperelliptic double cover
$\varphi_{|g^1_2|}\colon C\to \mP^1$ factors through the map 
$\varphi_{|g^{1}_{2}+\theta+p|}$. So we have only to show that 
$|\theta+g^{1}_{2}+p|$ separates the two points in a member of $|g^1_2|$.
This is equivalent to $h^0(\theta+g^{1}_{2}+p-g^1_2)=h^0(\theta+g^{1}_{2}+p)-2$, which follows from the above computations.
Since $C\to M$ is birational, the degree of $M$ is $g+2$.

\vspace{1pt}

\noindent (2). Since $h^0(\theta+p)=1$ as in the proof of (1), the linear system $|\theta+p|$ has a unique member $D$.
We see that $D$ is mapped to a point since
$\theta+g^{1}_{2}+p$ is the pull-back of $\sO_{\mP^2}(1)|_M$ and 
$h^0(\theta+g^1_2+p-(\theta+p))=h^0(g^1_2)=2$.

\vspace{1pt}

\noindent (3). The point $p$ is not contained in the support of $D$
since $h^0(\theta+p-p)=h^0(\theta)=0$. Let's us consider points of
$S\setminus \{p\}$. Without loss of generality,
we have only to show that $h^0(\theta+p-p_1)=0$.
By the Riemann-Roch theorem, the assertion is equivalent to $h^1(\theta+p-p_1)=0$.
By (\ref{eq:twopre}), $\theta+p-p_1=p_2+\cdots+p_{g+1}+p-g^1_2$.
Therefore, by the Serre duality,
we have
\[
h^1(\theta+p-p_1)=h^0(g\times g^1_2-(p_2+\cdots+p_{g+1}+p))
\]
since $K_C=(g-1)g^1_2$.
Now it is easy to verify this is zero by using the hyperelliptic morphism $C\to \mP^1$.

\vspace{1pt}

\noindent (4).
First we show that $t$ is different from the image of any point $x$ of $C\setminus D$.
Indeed, we have \[
h^0(\theta+g^1_2+p-(\theta+p)-x)=h^0(g^1_2-x)=1,
\] which means that $|\theta+g^1_2+p|$ separates $D$ and $x$.
In particular, we have the former assertion of (4) by (3). 

We show that $|\theta+g^1_2+p|$ separates any two of $p_1,\dots, p_{g+1}$.
Without loss of generality, we have only to consider the case of $p_1$ and $p_2$. It suffices to show that $h^0(\theta+g^1_2+p-p_1-p_2)=
h^0(\theta+g^1_2+p)-2=1$, which is equivalent to $h^1(\theta+g^1_2+p-p_1-p_2)=0$ by the Riemann-Roch theorem. 
By the presentation (\ref{eq:twopre}), we have
$h^1(\theta+g^1_2+p-p_1-p_2)=h^1(p_3+\cdots+p_{g+1}+p)$.
By the Serre duality, we have
\[
h^1(p_3+\cdots+p_{g+1}+p)=
h^0((g-1)g^1_2-p_3-\cdots-p_{g+1}-p)
\]
since $K_C=(g-1)g^1_2$.
Now it is easy to verify the r.h.s. is zero by using the hyperelliptic morphism $C\to \mP^1$.

The same argument shows that 
$|\theta+g^1_2+p|$ separates any two of $p'_1,\dots, p'_{g+1}$.
Moreover, if $p$ is distinct from a $p_i$ or $p'_j$,
the same proof works for the separation of $p$ and $p_i$ or $p'_j$.

It remains to show that 
$|\theta+g^1_2+p|$ separates one of $p_1,\dots,p_{g+1}$
and one of $p'_1,\dots, p'_{g+1}$.
Without loss of generality, 
we have only to consider the case of $p_1$ and $p'_1$.
If $p=p_1$, then $p\not =p'_1$, and hence we have already shown
that the images of $p=p_1$ and $p'_1$ are different. Thus we may assume
that $p\not =p_1, p'_1$. 
By (\ref{eq:twopre}),
$D_1:=p+p_1+\cdots+p_{g+1}$ and $D_2:=p+p'_1+\cdots+p'_{g+1}$ are two
distinct members of $|\theta+g^1_2+p|$.  
If the images of $p_1$ and $p'_1$ by the map $\varphi_{|\theta+g^{1}_{2}+p|}$ coincides,
then the images of $D_1$ and $D_2$ coincides since
they are the line through the images of $p$ and $p_1$, and
the line through the images of $p$ and $p'_1$.
This is a contradiction to a property of 
the map defined by $|\theta+g^1_2+p|$.
\end{proof}

\begin{proof}[{\bf Proof of Theorem \ref{Spinhyperellipticloci}}] 
Let $M$,
$r_1,\ldots, r_{g+1}$ and $r'_1,\ldots ,r'_{g+1}\in M$ be
the $\varphi_{|\theta+g^1_2+p|}$-images of $C$, the Weierstrass points 
$p_1,\dots, p_{g+1}$ and $p'_1,\dots, p'_{g+1}$
of $C$ as in (\ref{eq:twopre}), respectively. Let $r\in M$ be the image of $p$ and 
$t\in M$ the image of the unique member of $|\theta+p|$.
By Lemma \ref{lem:reconst} (4), 
$r$, $t$, $r_1,\ldots, r_{g+1}$, $r'_1,\ldots ,r'_{g+1}$ are distinct points
(recall that now we are assuming $p$ is not 
a Weierstrass point).
We set $V=H^0(C,\sO_C(\theta+g^{1}_{2}+p))^*$.
By (\ref{eq:twopre}), 
there are two lines $\ell,\ell'\subset \mP(V)$ such that
$\ell_{|M}=r_1+\cdots+r_{g+1}+r$ and
$\ell'_{|M}=r'_1+\cdots+r'_{g+1}+r$. 

We then identify the polarized 
    space $(\mP(V), \ell\cup \ell')$ with $(\Pst,\ell_1\cup \ell_2)$
(recall the notation as in Proposition \ref{prop:jump}).
   By this identification,
the point $r$ corresponds to $[\fj]$. Let $\fm$ be the line of $\mP^2$ such that $[\fm]$ corresponds to the point $t$. Since $r\not =t$, the line $\fm$ is not the jumping line $\fj$ of the bundle $\sE$ such that $B_a\simeq \mP(\sE)$.
Moreover, $\fm$ is not a jumping lines of the second kind of $\sE$, equivalently,
$[\fm]\not \in \ell_1\cup \ell_2$ since $t$ is distinct from $r$, $r_1,\ldots, r_{g+1}$, $r'_1,\ldots ,r'_{g+1}$.
This will show that $R$ constructed below satisfies 
Generality Condition
\ref{gen} (a).

We consider the linear system $|C_{0}(\fm)+(g+1)L_{|L_{\fm}}|$ 
on $L_{\fm}\subset B_a$. 
We look for a member $R\in |C_{0}(\fm)+(g+1)L_{|L_{\fm}}|$ 
with Generality Condition \ref{gen} (a)--(d) such that $C=C_R$. Note that the condition for an $R\in |C_{0}(\fm)+(g+1)L_{|L_{\fm}}|$ to intersect one fixed $B_a$-line is of codimension $1$.
Hence there exists at least one $R\in |C_{0}(\fm)+(g+1)L_{|{L_{\fm}}}|$ 
intersecting the $2g+2$ $B_a$-lines which correspond to the $2g+2$ points   
$r_1,\ldots, r_{g+1}$ and $r'_1,\ldots ,r'_{g+1}\in M$, since $\dim H^0(C_{0}(\fm)+(g+1)L_{|{L_{\fm}}})=2g+3$.
Equivalently, there exists 
at least one $R\in |C_{0}(\fm)+(g+1)L_{|L_{\fm}}|$ such that 
$r_1,\ldots ,r_{g+1}, r'_1,\ldots ,r'_{g+1}\in M_R$.
By Corollary \ref{cor:double}, $R$ intersects $F'_1$, and $F'_2$
at $g+1$ points, respectively, corresponding to $r_1,\ldots ,r_{g+1}$ and $r'_1,\ldots ,r'_{g+1}$.
Therefore, $R$ satisfies Generality Condition
\ref{gen} (d).
Moreover, $R$ does not pass through $F'_1\cap F'_2\cap L_{\fm}=\gamma_a\cap L_{\fm}$ since $r_1,\ldots, r_{g+1}$, $r'_1,\ldots ,r'_{g+1}$ are distinct points.
Therefore $R$ satisfies Generality Condition
\ref{gen} (c).

We show that $R$ is smooth, namely, $R$ satisfies 
Generality Condition \ref{gen} (b).
Indeed,
assume by contradiction that $R$ is reducible.
Then $R$ contains a ruling of $L_{\fm}$, say, $f$.
We have $f\cap \gamma_a=\emptyset$ since $R\cap \gamma_a=\emptyset$. 
Thus $M_R$ contains the curve $M_f$,
which is a line in $\Pst$ by Proposition \ref{trovato} (2), 
besides $M_f$ contains $t=[\fm]$, and one of 
$r_1,\ldots ,r_{g+1}$ and one of $r'_1,\ldots ,r'_{g+1}$
corresponding to $F'_1\cap f$ and $F'_2\cap f$, respectively.
By reordering the points, we may assume that $r_1, r'_1\in M_f$.
Therefore $t,r_1,r'_1$ are collinear.
This is, however, a contradiction since the line through $t$ and $r_1$
touches $M$ only at $t$ and $r_1$ (recall that $r_1$ is the image of a Weierstrass point).

Finally we show $M=M_R$. We have checked $\fm$ and $R$ satisfy Generality Condition
\ref{gen} (a)--(d).
Note that, 
by the constructions of $M$ and $M_R$ as the images of the map $\varphi_{|\theta+g^1_2+p|}$ and $\varphi_{|\theta_R+h_R+[\fj]_R|}$ respectively,
there exists a line 
through $t$ and touches both $M$ and $M_R$ at $r_i$ with multiplicity two $(i=1,\dots,g+1)$, 
and the same is true for $r'_j$ $(j=1,\dots,g+1)$.
Hence the intersection multiplicities of $M_R$ and $M$ at $r_i$ and $r'_j$ are
at least two.
Therefore the scheme theoretic intersection $M\cap M_R$ contains 
    $r$, the $2(g+1)$ points $r_{i},r'_{j}$, $i,j=1,... ,g+1$ with multiplicity $\geq 2$ and we also have a fat point of multiplicity $g^2$ at $t$. 
This implies that, if $M\not =M_R$, then $M\cdot M_R\geq 
    1+4(g+1)+g^{2}=(g+2)^{2}+1$, which is a contradiction since $\deg M=\deg M_R=g+2$. Now we conclude that $M_R=M$.
\end{proof}

Theorem \ref{Spinhyperellipticloci} has a nice corollary,
which seems to be unknown.

\begin{cor}\label{belloiperellitticopuntato} The moduli space
$\sS^{0,{\rm{hyp}}}_{g,1}$ and the moduli space
$\sS^{0,{\rm{hyp}}}_{g}$ of ineffective spin hyperelliptic curves are irreducible. 
\end{cor}

\begin{proof} 
By Definition \ref{proclaimer},
$\sH_{g+2}$ is an open subset of the projective bundle $\Sigma_{g+2}$ over the projective plane. Therefore $\sH_{g+2}$ is irreducible.
By Theorem \ref{Spinhyperellipticloci} we know that 
    the map $ \pi_{g,1}\colon  
\sH_{g+2}\dashrightarrow \sS^{0,{\rm{hyp}}}_{g,1}$ is dominant
to each irreducible component of $\sS^{0,{\rm{hyp}}}_{g,1}$. 
The forgetful morphism 
$\sS^{0,{\rm{hyp}}}_{g,1} \to\sS^{0,{\rm{hyp}}}_{g}$ is dominant too. 
Hence the claim follows.

\end{proof}

\subsection{Birational model of $\sS^{0,{\rm{hyp}}}_{g,1}$}
\label{subsection:Bir}
Let $\fm$ be a general line in $\mP^2$.
By Theorem \ref{Spinhyperellipticloci} and the group action of $G$ on $\sH_{g+2}$, 
the map $ \pi_{g,1}\colon  
\sH_{g+2}\dashrightarrow \sS^{0,{\rm{hyp}}}_{g,1}$ induces a dominant rational map $\rho_{g,1}\colon |(H+gL)_{|L_{\fm}}|\dashrightarrow
\sS^{0,{\rm{hyp}}}_{g,1}$. 
Recall the definition of the subgroup $\Gamma$ of $G$ as in Lemma \ref{lem:fixLm}.
By the classical Rosenlicht theorem, we can find an $\Gamma$-invariant open set $U$ of $|(H+gL)_{|L_{\fm}}|$ such that the quotient $U/\Gamma$ exists.
Since a general $\Gamma$-orbit in $|(H+gL)_{|L_{\fm}}|$ is mapped to a point by
$\rho_{g,1}$,
we obtain a dominant map $\overline{\rho}_{g,1}\colon U/\Gamma\to \sS^{0,{\rm{hyp}}}_{g,1}$. 

\begin{prop}
\label{prop:bir}
The dominant map $\overline{\rho}_{g,1}\colon U/\Gamma\to \sS^{0,{\rm{hyp}}}_{g,1}$
is birational. 
\end{prop}

\begin{proof}
We show that 
$\overline{\rho}_{g,1}$ is generically injective. 
We consider two general elements 
$R,R'\in U$ and the two corresponding $\Gamma$-orbits
$\Gamma[R],\Gamma[R']$. Note that $M_R$ and $M_{R'}$ both pass through the points $[\fj]$ and $[\fm]$, and they both have Weierstrass points distributed on the two lines $\ell_1$ and $\ell_2$. Now assume that $[C_R,p,\theta_R]=[C_{R'},p',\theta_{R'}]\in \sS^{+{\rm{hyp}}}_{g,1}$, equivalently, there exists an isomorphism $\xi\colon C_R\to C_{R'}$ such that $\xi^*\theta_{R'}=\theta_R$ and $\xi(p)=p'$. 
We consider the following diagram:
\[
\xymatrix{C_R\ar[rr]^{(b\circ \tilde{p}_2)|_{C_R}}\ar[d]_{\xi} & & M_R\\
C_{R'}\ar[rr]_{(b\circ \tilde{p}_2)|_{C_{R'}}} & & M_{R'}.}
\]
Note that $(b\circ \tilde{p}_2)|_{C_R}(p)=
(b\circ \tilde{p}_2)|_{C_R}(p')=[\fj]$ by Notation \ref{nota:marked}.
Since the $g^1_2$ is unique on an hyperelliptic curve, 
we have $\xi^*h_{R'}=h_R$ where $h_R$ and $h_{R'}$ are respectively the $g^1_2$'s of $C_R$ and $C_{R'}$. 
Therefore there exists 
a projective isomorphism $\xi_M$ from $M_R$ to $M_{R'}$
such that $(b\circ \tilde{p}_2)|_{C_{R'}}\circ \xi=
\xi_M\circ (b\circ \tilde{p}_2)|_{C_{R}}$ and hence $\xi_M([\fj])=[\fj]$
since the morphisms $(b\circ \tilde{p}_2)|_{C_{R}}\colon C_R\to M_R\subset\Pst$ and $(b\circ \tilde{p}_2)|_{C_{R'}} \colon C_{R'}\to M_{R'}\subset\Pst$ are given respectively by $|\theta_R+p+h_R|$ and $|\theta_{R'}+p'+h_{R'}|$.
We also have $\xi_M([\fm])=[\fm]$ since $[\fm]$ is a unique $g$-ple point of
$M_R$ and $M_{R'}$ respectively by Proposition \ref{trovato} (3) and (4).
Let $g$ be an element of $\Aut \Pst$ inducing the projective isomorphism $\xi_M$. Since $\xi$ sends the Weierstrass points of $C_R$ to those of $C_{R'}$, the line pair $\ell_1\cup \ell_2$ must be sent into itself by $g$. Hence $g\in G$. Moreover, since $g$ fixes $[\fm]$ as we noted above, we have $g\in \Gamma$. In summary, we have shown
$gM_R=M_{R'}$. It remains to show that $gR=R'$. 
For this, we have only to show that $R$ is recovered from $M_R$.
Take a general line $\ell$ through $[\fm]$ and set $\ell_{|M_R}=[\fm]+[C_1]+[C_2]$ set-theoretically,
where $C_1$ and $C_2$ are $B_a$-lines. Note that $C_1\cap C_2$ is one point. 
Then $R$ is recovered as the closure of the locus of $C_1\cap C_2$
when $\ell$ varies.
\end{proof}

\section{Proof of Rationality}
\label{section:Rat}

\begin{thm}\label{belloiperellitticorat}
   $\sS^{0,{\rm{hyp}}}_{g,1}$ is a rational variety.
 \end{thm}
 \begin{proof} 
As in the subsection \ref{subsection:Bir}, we fix a general line $\fm$ in $\mP^2$.
By Proposition \ref{prop:bir}, we have only to show that 
$U/\Gamma$ is a rational variety.

Using the elementary transformation as in Proposition \ref{prop:elm}, 
we are going to reduce the problem to that on 
$\mP^1\times \fm$. Let $r_v$ and $r_h$ are rulings of 
the projections $\mP^1\times \fm\to \fm$ and 
$\mP^1\times \fm\to \mP^1$, respectively.
From now on, we identify $\mP^1\times \fm$ with $\mP^1\times \mP^1$ 
having the bi-homogeneous coordinate $(x'_1:x'_2)\times (y_2:y_3)$
with $x'_1:=(x_1-x_2)/{2}$ and $x'_2:=(x_1+x_2)/{2}$.
To clarify the difference of the two factors of $\mP^1\times \mP^1$,
we keep denoting it by $\mP^1\times \fm$.
With this coordinate, the action of 
$\Gamma\simeq (\mZ_2\times G_a)\rtimes G_m$ on $\mP^1\times \fm$
is described by multiplications of the following matrices
by Lemma \ref{lem:fixLm}:
\begin{itemize}
\item $G_m:$
{\scriptsize{$\begin{pmatrix} 1 & 0\\ 0 & 1\end{pmatrix}\times
\begin{pmatrix} 1& 0 \\
0 & a 
\end{pmatrix}$}}
with $a\in G_m$,
\item $G_a:$
{\scriptsize $\begin{pmatrix} 1 & 0\\ 0 & 1\end{pmatrix}\times
\begin{pmatrix} 1 & b\\
0 & 1 
\end{pmatrix}$}
with $b\in G_a$, and
\item $\mZ_2:$
\scriptsize{$\begin{pmatrix} -1& 0\\ 0 & 1\end{pmatrix}\times
\begin{pmatrix} 
1 & 0 \\
0 & 1  
\end{pmatrix}.$}

\end{itemize}
Note that members of $|(H+gL)_{|L_{\fm}}|$
corresponds to those of the linear system $|r_h+(g+1)r_v|$
through the point $c:=\gamma_c\cap (\mP^1\times \fm)=(1:0)\times (1:0)$.
We denote by $\Lambda$ the sublinear system consisting of such members.
A member of $\Lambda$ is the zero set of a bi-homogeneous polynomial of 
bidegree $(1,g+1)$
of the form $x'_1 f_{g+1}(y_2,y_3)+x'_2 g_{g+1}(y_2,y_3)$, where
$f_{g+1}(y_2,y_3)$ and $g_{g+1}(y_2,y_3)$ are binary $(g+1)$-forms
\begin{align*}
f_{g+1}(y_2,y_3)&=p_{g} y_2^{g}y_3+\cdots+p_i y_2^i y_3^{g+1-i}+\cdots+p_0 y_3^{g+1},\\
g_{g+1}(y_2,y_3)&=q_{g+1} y_2^{g+1}+\cdots+q_i y_2^i y_3^{g+1-i}+\cdots+q_0 y_3^{g+1}.
\end{align*}
Then the linear system $\Lambda$ can be identified with 
the projective space $\mP^{2g+2}$ with the homogeneous coordinate
$(p_0:\cdots:p_{g}:q_0:\cdots:q_{g+1})$.
A point $(p_0:\cdots: p_i:\cdots:p_{g}:q_0:\cdots:q_j:\cdots:q_{g+1})$ is mapped by the elements of the subgroups $G_m$, $G_a$, and $\mZ_2\subset \Gamma$ as above to the following points:
\begin{enumerate}[(a)]
\item $G_m:$
$(a^{g+1} p_0:\cdots:a^{g+1-i} p_i:\cdots:a p_{g}:a^{g+1} q_0:\cdots:a^{g+1-j} q_j:\cdots: q_{g+1})$,
\item $G_a:$ the point
$(p'_0:\cdots: p'_i:\cdots:p'_{g}:q'_0:\cdots:q'_j:\cdots:q'_{g+1})$ with 
\begin{align}
p'_i&=\sum_{k=i}^{g}\binom{k}{i} b^{k-i} p_{k},\label{align:p'q'}\\
q'_j&=\sum_{l=j}^{g+1}\binom{l}{j} b^{l-j} q_{l},\nonumber
\end{align}
\item $\mZ_2:$
$(-p_0:\cdots:-p_i:\cdots:-p_{g}:q_0:\cdots:q_j:\cdots:q_{g+1})$.
\end{enumerate}   

\vspace{2pt}

\noindent {\bf Step 1.} The quotient $\Lambda_1:=\Lambda/\mZ_2$ is rational.

The rationality is well-known by the description of $\mZ_2$-action as in (c).
In the following steps, it is convenient to show this more explicitly.
On the open set $\{q_{g+1}\not =0\}\subset \Lambda$, which is $\Gamma$-invariant, we may consider $q_{g+1}=1$. Then the action is 
\[
(p_0,\cdots,p_{g},q_0,\cdots,q_{g})\mapsto
(-p_0,\cdots,-p_{g},q_0,\cdots,q_{g}).
\] 
Therefore the quotient map can be written on the $\Gamma$-invariant open subset $\{p_{g}\not =0\}$ as follows:
\[
(p_0,\cdots,p_i,\cdots, p_{g},q_0,\cdots,q_{g})\mapsto
(p_0 p_{g},\cdots,p_i p_{g}, \cdots, p_{g}^2,q_0,\cdots,q_{g}).
\]
We denote by $\tC^{2g+2}$ the target $\mC^{2g+2}$ of this map and 
by $(\tilde{p}_0,\dots, \tilde{p}_{g},\tilde{q}_0,\dots,\tilde{q}_{g})$ its coordinate.
Using this presentation, we compute the quotient by the additive group $G_a$ in the next step.

\vspace{2pt}

\noindent {\bf Step 2.} The quotient $\Lambda_2:=\Lambda_1/G_a$ is rational.

Let $(\tilde{p}'_0,\dots, \tilde{p}'_{g},\tilde{q}'_0,\dots,\tilde{q}'_{g})$be the image of the point $(\tilde{p}_0,\dots, \tilde{p}_{g},\tilde{q}_0,\dots,\tilde{q}_{g})$ by the action of an element of $G_a$ as in (b).
By the choice of coordinate, it is easy to check $\tilde{p}'_i$ and $\tilde{q}'_j$ can be written by $\tilde{p}_0,\dots, \tilde{p}_{g}$ and 
$\tilde{q}_0,\dots, \tilde{q}_{g}$ respectively by the formulas
obtained from (\ref{align:p'q'}) by setting $q_{g+1}=1$ and replacing
$p'_i$, $p_{k}$, $q'_j$ and $q_{l}$ with
$\tilde{p}'_i$, $\tilde{p}_{k}$, $\tilde{q}'_j$ and $\tilde{q}_{l}$.
Then note that we have $\tilde{q}'_{g}=\tilde{q}_{g}+(g+1)b$.
Therefore, 
the stabilizer group of every point is trivial and
every $G_a$-orbit intersects the closed set $\{\tilde{q}_{g}=0\}$ at a single point. 
Hence we may identified birationally the quotient $\tC^{2g+2}/G_a$ with
the closed set $\{\tilde{q}_{g}=0\}\subset \tC^{2g+2}$. In particular, the quotient is rational.

\vspace{2pt}

\noindent {\bf Step 3.} The quotient $\Lambda_3:=\Lambda_2/G_m$ is rational.

We may consider the closed set 
$\{\tilde{q}_{g}=0\}$ as 
the affine space $\mC^{2g+1}$ with the coordinate
$(\tilde{p}_0,\dots,\tilde{p}_{g},\tilde{q}_0,\dots,\tilde{q}_{g-1})$.
Note that this closed set has the naturally induced $G_m$-action
such that, by the element of $G_m$ as in (a),
a point
$(\tilde{p}_0,\dots,\tilde{p}_{g},\tilde{q}_0,\dots,\tilde{q}_{g-1})$
is mapped to
$(a^{g+2} \tilde{p}_0,\dots,a^2 \tilde{p}_{g},a^{g+1}\tilde{q}_0,\dots,a^2\tilde{q}_{g-1})$.
Therefore the quotient $\mC^{2g+1}/G_m$ is a weighted projective space, hence is rational.

 \end{proof}

\end{document}